\documentclass[12pt,a4paper]{amsart}
\usepackage{graphicx}
\usepackage{amsmath,amssymb,amsthm,amsfonts,mathrsfs,enumerate,sansmath,mathtools}

\newtheorem{proposition}{Proposition}[section]
\newtheorem{theorem}{Theorem}[section]

\newtheorem{lemma}{Lemma}[section]
\newtheorem{corollary}{Corollary}[section]

\usepackage{tikz-cd}
\newcommand{\n}{\noindent}
\usepackage{color}
\allowdisplaybreaks
\begin{document}
\title{Gyrogroups associated with groups}
\author{Ratan Lal and Vipul Kakkar}
\address{Department of Mathematics, Central University of Rajasthan, Rajasthan}
\email{vermarattan789@gmail.com, vplkakkar@gmail.com}
\date{}

\begin{abstract}
\n In this paper, we study the properties of the associated gyrogroup ${^\circ}G$ of a given group $G$ of nilpotency class $3$. We have proved that if $3$ does not divide the order of the group $G$, then the nilpotency class of the associated gyrogroup ${^\circ}G$ is same as that of the group $G$. We have also studied the problem of abelian inner mapping group in this context. 
\smallskip

\n \textbf{Keywords:} Gyrogroups, Nilpotent Groups, Inner Mapping Groups.
\end{abstract}

\maketitle
\section{Introduction}
\n The study of gyrogroups was initiated by Ungar in \cite{ung}. Gyrogroups are grouplike structures and non-associative generalization of groups. A groupoid $(L,\cdot)$ with identity is called a right loop if the equation $X a=b$ has a unique solution in $L$, for all $a,b\in L$. Let $L$ be a right loop and $y,z\in L$. Then, there is a bijective map $f(y,z)$ from $L$ to $L$ such that 
\[(xy)z=f(y,z)(x)(yz),\; \text{for all}\;x\in L.\]

\n A right loop $L$ is called a loop if the equation $a X = b$ has a unique solution in $L$. A right loop $L$ is called a gyrogroup if $f(a,b)=f(ab,a)^{-1}$ and $f(a,b)$ are automorphisms of $L$ for all $a,b\in L$. Gyrogroups are loops. Let $G$ be a group. Define a binary operation $\circ$ on $G$ by $x\circ y=y^{-1}xy^2$. Foguel and Ungar in \cite{foguel2001gyrogroups} proved that $(G,\circ)$ is a gyrogroup if and only if $G$ is central by a $2$-Engel group (see \cite[Theorem 3.7]{foguel2001gyrogroups}). In particular if $G$ is nilpotent group of class $3$, then $(G, \circ)$ is a gyrogroup. It is also shown that the associated right gyrogroup is a group if and only if the group $G$ is nilpotent group of class $2$ (see \cite[Theorem 3.6]{foguel2001gyrogroups}\label{s1t1}). We will denote the associated gyrogroup $(G,\circ)$ by $^{\circ}G$. Throughout the paper, $G$ will denote the finite nilpotent group of class $3$, otherwise will be stated separately.

\vspace{0.2 cm}

\n Let $L$ be a loop and $N$ be a subloop of $L$. Then $N$ is called a normal subloop of the loop $L$ if, for all $x, y \in L$, we have, $(i)$ $x N = N x$, $(ii)$ $x(y N)= (x y) N$ and $(iii)$ $(N x) y = N(xy)$. In a loop $L$, we have the following important subsets:
\begin{enumerate}
	\item [$(i)$]The set $N_{\lambda}(L) = \{a \in L \mid (a x) y = a(x y)\;\; \forall x, y \in L\}$ is called the left nucleus.
	\item[$(ii)$]The set $N_{\mu}(L) = \{a \in L \mid (x a) y = x(a y)\;\;\forall x, y \in L\}$ is called the middle nucleus.
	\item[$(iii)$]The set $N_{\rho}(L) = \{a \in L \mid (x y) a = x(y a)\;\;\forall x, y \in L\}$ is called the right nucleus.
	\item [$(iv)$] The set $N(L) = N_{\lambda}(L)\cap N_{\mu}(L) \cap N_{\rho}(L)$ is called the nucleus of $L$.
	\item [$(v)$] The set $C(L) = \{a\in L \mid x a = a  x \;\;\forall x \in L\}$ is called the commutant of $L$.
	\item [$(vi)$] The set $Z(L) = C(L)\cap N(L)$ is called the center of $L$.
\end{enumerate}  
The commutant $C(L)$ is not necessarily a subloop of $L$. Except this, all are subloops of $L$ in the above list. The center $Z(L)$ is the normal subloop of $L$. Let $x, y, z \in L$. Then, the commutator $[x,y]$ and the associator $A(x,y,z)$ are defined as the unique solutions of the following equations respectively,
\begin{eqnarray*}
xy &=& [x,y](yx),\\
\text{and}\;(xy)z &=& A(x,y,z)(x(yz)).
\end{eqnarray*}

\n Let $[L,L]$ and $\mathcal{A}(L)$ denotes the collection of all the commutators and the associators of the loop $L$. The commutator of the associated gyrogroups will be denoted by $^{\circ}[x,y]$. $[G,G,G]$ denotes the subgroup of the group $G$ generated by the triple commutators $[x,y,z]=[[x,y],z]$, for all $x,y,z\in G$. 

\vspace{0.2 cm}

\n In section $3$, we have studied the properties of nuclei, commutant and center of the associated gyrogroup. Moreover, we have proved that the commutant $C({^\circ}G)$, the center $Z({^\circ}G)$ of the loop ${^\circ}G$ and the center $Z(G)$ of the group $G$ all coincide, if $3$ does not divide the order of $G$. In section $4$, we have studied the nilpotency class of associated gyrogroup. We have proved that, if $3$ does not divide the order of the group $G$, then ${^\circ}G$ is a nilpotent loop of class $3$. In section $5$, we have studied the open problem for abelian inner mapping group for loop of class $3$ in case of associated gyrogroups. 

\section{Nuclei, Commutant and Center of the Associated Gyrogroup }
\n In this section, we prove some properties of the nuclei, commutant and the center of the associated gyrogroup. By \cite[Proposition 4.4]{lal2014weak}, the commutant $C(^{\circ}G)$ is the characteristic subgroup of the group $G$. We prove below that it is a normal subloop of $^{\circ}G$. 

	\begin{proposition}\label{sl}
		The commutant $C({^\circ}G)$ is a subloop of the loop ${^\circ}G$.
	\end{proposition}
	\begin{proof}
		Let $a\in G$ be any element. Then, from the proof of the Proposition \cite[Proposition 4.4, p. 1450058-10]{lal2014weak}, we have,
		\begin{equation}
		a\in C(^{\circ}G) \iff (ax)^{3}= a^{3}x^{3}\;\; \forall x\in G. \label{s2e1}
		\end{equation}
		First, we prove that $C({^\circ}G)$ is closed under the binary operation $\circ$. Let $a, b \in C({^\circ}G)$. Then for all $x \in {^\circ}G$, we have
		\begin{eqnarray*}
			(a\circ b)^{3}x^{3} &=& (b^{-1}ab^{2})^{3}x^{3}\\
			&=& ((b^{-1}ab)b)^{3}x^{3}\\
			&=& ((b^{-1}ab)^{3}b^{3})x^{3}, (\textnormal{using (\ref{s2e1})}) \\
			&=& (b^{-1}ab)^{3}(bx)^{3}, (\textnormal{using (\ref{s2e1})}) \\
			&=& ((b^{-1}ab)(bx))^{3}, (\textnormal{using (\ref{s2e1})}) \\
			&=& (b^{-1}ab^{2}x)^{3}\\
			&=& ((a\circ b)x)^{3}.
		\end{eqnarray*}
		Therefore, $a\circ b \in C({^\circ}G)$.

		\vspace{.2cm}
		
		\n 	Let $a, b \in C({^\circ}G)$. Note that, $aba^{-2}\in C({^\circ}G)$ and	$aba^{-2}\circ a = a^{-1}aba^{-2}a^{2} = b$. Thus, for any two elements $a,b \in C({^\circ}G)$, the equation $X \circ a = b$ has a unique solution in $C({^\circ}G)$. Therefore, $C({^\circ}G)$ is a subloop of ${^\circ}G$.
	\end{proof}

\begin{proposition}
	The commutant $C({^\circ}G)$ is a normal subloop in ${^\circ}G$. 
\end{proposition} 
\begin{proof}
	\n $(i)$ Clearly, $x\circ C({^\circ}G) = C({^\circ}G)\circ x$ for all $x \in {^\circ}G$.
	
		\vspace{.2cm}
		
		\n $(ii)$ For all $a\in C({^\circ}G)$ and $x,y \in {^\circ}G$,
		
		\begin{eqnarray*}
			a\circ (x\circ y) &=& a\circ (y^{-1}xy^{2})\\
			&=& y^{-2}x^{-1}yay^{-1}xy^{2}y^{-1}xy^{2}\\
			&=& y^{-1}(y^{-1}x^{-1}yay^{-1}xy)xy^{2}\\
			&=& y^{-1}x^{-1}(xy^{-1}x^{-1}y)a(y^{-1}xyx^{-1})x^{2}y^{2}\\
			&=& y^{-1}x^{-1}[x,y^{-1}]a[y^{-1}, x]x^{2}y^{2}\\
			&=& y^{-1}(x^{-1}bx^{2})y^{2}\\
			&=& (b\circ x)\circ y.
		\end{eqnarray*}
		Thus, $a\circ (x\circ y) = (b\circ x)\circ y$, where $b= [x,y^{-1}]a[y^{-1}, x]$. Since $C({^\circ}G)$ is a normal subgroup of the group $G$, $b\in C({^\circ}G)$. Therefore, $C({^\circ}G)\circ (x\circ y) = (C({^\circ}G)\circ x)\circ y$, for all $x, y \in {^\circ}G$.
		
		\vspace{.2cm}
		
		\n $(iii)$ Using $(i)$ and $(ii)$, we have $x\circ (y\circ C({^\circ}G)) = (x\circ y)\circ C({^\circ}G) \iff x\circ (C({^\circ}G) \circ y) = C({^\circ}G)\circ (x\circ y) \iff x\circ (C({^\circ}G) \circ y) = (C({^\circ}G)\circ x)\circ y \iff x\circ (C({^\circ}G) \circ y) = (x\circ C({^\circ}G))\circ y$. Note that, for all $a\in C({^\circ}G)$ and $x\in {^\circ}G$, $(ax)^3=a^3x^3 \iff (ax)^{2} = x^{2}a^{2}$. Now, for all $a \in C({^\circ}G)$ and $x, y\in {^\circ}G$, we have,
		\begin{eqnarray*}
			x\circ (a\circ y)&=&x\circ (y^{-1}ay^{2})\\
			&=&  y^{-2}a^{-1}yxy^{-1}ayay^{2}\\
			&=& y^{-2}a^{-1}yxy^{-1}(ay)^{2}y\\
			&=& y^{-2}a^{-1}yxy^{-1}y^{2}a^{2}y\\
			&=& y^{-2}(a^{-1}yxya^{2})y\\
			&=& y^{-2}(yxy \circ a)y\\
			&=& y^{-2}(a\circ yxy)y\\
			&=& y^{-2}y^{-1}x^{-1}y^{-1}ayxyyxyy\\
			&=& y^{-3}x^{-1}y^{-1}ayxy^{2}xy^{2}\\
			&=& y^{-1}x^{-1}(xy^{-2}x^{-1}y^{-1})a(yxy^{2}x^{-1})x^{2}y^{2}\\
			&=& y^{-1}x^{-1}(yxy^{2}x^{-1})^{-1}a(yxy^{2}x^{-1})x^{2}y^{2}\\
			&=& y^{-1}(x^{-1}bx^{2})y^{2}\\
			&=& (b\circ x)\circ y.
		\end{eqnarray*}
		\n	Thus, $x\circ (a \circ y) = (x\circ b)\circ y$, where $b= (yxy^{2}x^{-1})^{-1}a(yxy^{2}x^{-1})$. Since $C({^\circ}G)$ is a normal subgroup of the group $G$, $b\in C({^\circ}G)$. Therefore, $x\circ (y\circ C({^\circ}G)) = (x\circ y)\circ C({^\circ}G)$ for all $x, y \in {^\circ}G$.
	
	\vspace{.2cm}
	
\n	Hence, $C({^\circ}G)$ is a normal subloop of the loop ${^\circ}G$.
\end{proof}
\begin{lemma}\label{l3}
	Let $G$ be a group. Then for all $x,y,z \in G$ following holds,
	\begin{enumerate}
		\item[($i$)] $[xy, z] = [x, [y,z]][y,z][x,z],$
		\item[($ii$)] $[x,yz] = [x,y][y,[x,z]][x,z]$.
	\end{enumerate}
\end{lemma}

\n Now, we will prove that the left, the middle and the right nuclei of $^{\circ}G$ are characteristic subgroups of the group $G$. First, note that, the center $Z(G)$ of the group $G$ is contained in the left, the middle and  the right nuclei of $^{\circ}G$.
	\begin{proposition}\label{ln}
	The left nucleus $N_{\lambda}({^\circ}G)$ is a subgroup of the group $G$. 
\end{proposition}
\begin{proof}
	Let $x,y\in {^\circ}G$. Then $a\in N_{\lambda}({^\circ}G)$
	
	$\iff a\circ (x\circ y) = (a\circ x)\circ y$
	
	$\iff a\circ (y^{-1}xy^{2}) = y^{-1}(a\circ x)y^{2}$
	
	$\iff (y^{-1}xy^{2})^{-1}a(y^{-1}xy^{2})^{2} = y^{-1}x^{-1}ax^{2}y^{2}$
	
	$\iff y^{-2}x^{-1}yay^{-1}xy^{2}y^{-1}xy^{2} = y^{-1}x^{-1}ax^{2}y^{2}$
	
	$\iff y^{-1}x^{-1}yay^{-1}xy = x^{-1}ax$
	
	$\iff xy^{-1}x^{-1}ya = axy^{-1}x^{-1}y$
	
	$\iff [x,y^{-1}]a = a[x,y^{-1}]$
	
	$\iff [a,[x,y^{-1}]] = 1$.
	
	\vspace{.2cm}
	
	\n Now, let $a,b\in N_{\lambda}({^\circ}G)$. Then, using the Lemma \ref{l3} $(i)$, 	
	\begin{eqnarray*}
	[ab, [x,y^{-1}]] &=& [a, [b,[x,y^{-1}]]][b, [x,y^{-1}]][a, [x,y^{-1}]] \\
	&=& 1\; (\text{for } [G,G,G]\subseteq Z(G)).
	\end{eqnarray*}
	
	\n Thus $ab\in N_{\lambda}({^\circ}G)$.  Also, for $a\in N_{\lambda}({^\circ}G)$ and $x,y\in {^\circ}G $, 
	
	\begin{align*}
	1= & [1, [x, y^{-1}]] \\
	 = & [aa^{-1},[x,y^{-1}]] \\
	 = & [a, [a^{-1},[x,y^{-1}]]][a^{-1},[x,y^{-1}]][a,[x,y^{-1}]] \\
	 = & [a^{-1},[x,y^{-1}]].
	\end{align*}
	
	\n Therefore, $a^{-1}\in N_{\lambda}({^\circ}G)$. Hence, $N_{\lambda}({^\circ}G)$ is a subgroup of the group $G$. 
	\end{proof}
\begin{proposition}
	The middle nucleus $N_{\mu}({^\circ}G)$ is a subgroup of $G$.
\end{proposition}
\begin{proof}
	Let $x,y\in {^\circ}G$. Then, by the similar argument as in the proof of Proposition \ref{ln}, 
	\[a\in N_{\mu}({^\circ}G)\iff [x,[a,y^{-1}]] = 1.\]	
	
	\n Now, for all $a,b\in N_{\mu}({^\circ}G)$, by the similar argument as in the proof of Proposition \ref{ln}, using the Lemma \ref{l3} $(i)$, we have, $[x,[ab,y^{-1}]] = 1$ and $[x, [a^{-1},y^{-1}]] = 1$. Thus $a^{-1}, ab\in N_{\mu}({^\circ}G)$. Hence, $N_{\mu}({^\circ}G)$ is a subgroup of $G$.	
\end{proof}
\begin{proposition}
	The right nucleus $N_{\rho}({^\circ}G)$ is a subgroup of $G$.
\end{proposition}
\begin{proof}
	Let $x,y\in {^\circ}G$. Then, by the similar argument as in the proof of Proposition \ref{ln},
	\[a\in N_{\rho}({^\circ}G)\iff [x,[y,a^{-1}]] = 1.\]

	\n By the similar argument as in the proof of Proposition \ref{ln}, one can show that $N_{\rho}({^\circ}G)$ is a subgroup of $G$. 
\end{proof}
\begin{proposition}\label{nlpt}
	$N_{i}({^\circ}G)$, where $i\in \{\lambda, \mu, \rho\}$ are characteristic subgroups of the group $G$ and are of nilpotency class atmost 2. 
\end{proposition}
\begin{proof}
	Let $\psi \in Aut(G,\cdot)$ and $a\in N_{\lambda}({^\circ}G)$. Then for all $x,y\in G$, we have
	\begin{eqnarray*}
		(\psi(a)\circ x)\circ y &=& (\psi(a)\circ \psi(u))\circ \psi(v), \text{where}\; x=\psi(u),\; \text{and}\; y=\psi(v)\\
		&=& \psi(v)^{-1}(\psi(u)^{-1}\psi(a)\psi(u)^{2})\psi(v)^{2}\\
		&=&\psi(v^{-1}(u^{-1}au^{2})v^{2})\\
		&=& \psi((a\circ u)\circ v)\\
		&=& \psi(a\circ(u\circ v))\\
		&=& \psi(a\circ(v^{-1}uv^{2}))\\
		&=& \psi((v^{-1}uv^{2})^{-1}a(v^{-1}uv^{2})^{2})\\
		&=& \psi(v^{-1}uv^{2})^{-1}\psi(a)\psi(v^{-1}uv^{2})^{2}\\
		&=& (\psi(v)^{-1}\psi(u)\psi(v)^{2})^{-1}\psi(a)(\psi(v)^{-1}\psi(u)\psi(v)^{2})^{2}\\
		&=& (\psi(u)\circ \psi(v))^{-1}\psi(a)(\psi(u)\circ \psi(v))^{2}\\
		&=& \psi(a)\circ (\psi(u)\circ \psi(v))\\
		&=& \psi(a)\circ (x\circ y).	
	\end{eqnarray*}
\n Thus, $\psi(a)\in N_{\lambda}({^\circ}G)$. Hence, $N_{\lambda}({^\circ}G)$ is a characteristic subgroup of $G$.	Also, by the proof of Proposition \ref{ln}, one notes that, $[N_{\lambda}({^\circ}G), N_{\lambda}({^\circ}G),$ $N_{\lambda}({^\circ}G)] = \{1\}$. Hence, $N_{\lambda}({^\circ}G)$ is a nilpotent group of class atmost 2.	By the similar argument, $N_{\mu}({^\circ}G)$ and $N_{\rho}({^\circ}G)$ are characteristic subgroups of the group $G$ of nilpotency class atmost 2.
\end{proof}
\begin{corollary}\label{nug}
	$(N_{i}({^\circ}G), \circ)$, where $i\in \{\lambda, \mu, \rho\}$ are groups with the induced binary operation $\circ$.
\end{corollary}
\begin{proof}
	Follows immediately by the Propositions \ref{nlpt} and the Theorem \cite[Theorem 3.6]{foguel2001gyrogroups}. 
\end{proof}
\begin{proposition}\label{lmrn}
The following relations hold between the nuclei of $^{\circ}G$:
	\begin{itemize}
		\item[$(i)$] $N_{\mu}({^\circ}G) = N_{\rho}({^\circ}G),$
		\item[($ii$)] $N_{\mu}({^\circ}G)\subseteq N_{\lambda}({^\circ}G)$.
	\end{itemize}
\end{proposition}
\begin{proof}	
	\n $(i)$ Let $a\in N_{\rho}({^\circ}G)$. Then, for all $x,y\in {^\circ}G,$
		\begin{eqnarray*}
			[x,[a,y^{-1}]] &=& [x,[y^{-1},a]^{-1}]\\
			&=& [x,[y^{-1},a]]^{-1},\; \text{as,} \; [G,G,G]\subseteq Z(G)\\
			&=& 1, \; \text{as}\; a\in N_{\rho}({^\circ}G).
			\end{eqnarray*}
			
	\n	Thus, $[x,[a,y^{-1}]] = 1$, for all $x,y\in {^\circ}G$. Therefore, $a\in N_{\mu}({^\circ}G)$ and $N_{\rho}({^\circ}G)\subseteq N_{\mu}({^\circ}G)$.\\
	
\n	Conversely, let $a\in N_{\mu}({^\circ}G)$. Then, for all $x,y\in {^\circ}G$,
\begin{eqnarray*}
[x,[y,a^{-1}]] &=& [x, [a^{-1}, y]^{-1}]\\
&=& [x, [a^{-1},y]]^{-1}\; \text{as,} \; [G,G,G]\subseteq Z(G)\\
&=& 1, \; \text{as}\; a^{-1}\in N_{\mu}({^\circ}G).
\end{eqnarray*}
\n Therefore, $a\in N_{\rho}({^\circ}G)$ and $N_{\mu}({^\circ}G)\subseteq N_{\rho}({^\circ}G)$.
Hence, $N_{\mu}({^\circ}G) = N_{\rho}({^\circ}G)$.

\vspace{0.2 cm}

		\n $(ii)$ Let $a\in N_{\mu}({^\circ}G)$. Then, for all $x,y\in {^\circ}G$,
		\begin{align*}
			[a,[x,y]] =& a[x,y]a^{-1}[x,y]^{-1}\\
			=& a[y,a^{-1}]^{-1}([y,a^{-1}][x,y]a^{-1}yx)y^{-1}x^{-1}\\
			=& a[a^{-1},y]([x,y][y,a^{-1}]a^{-1}yx)y^{-1}x^{-1} (\text{ for } [G,G] \text{ is abelian})\\
			=& yay^{-1}(xyx^{-1}y^{-1}ya^{-1}x)y^{-1}x^{-1}\\
			=&  yay^{-1}a^{-1}x(x^{-1}axyx^{-1}a^{-1}x)y^{-1}x^{-1}\\
			=& yay^{-1}a^{-1}x^{2}(x^{-2}axyx^{-1}a^{-1}x^{2}x^{-1}a^{-1}x^{2})x^{-2}ay^{-1}x^{-1}\\
			=& yay^{-1}a^{-1}x^{2}((x^{-1}a^{-1}x^{2})^{-1}y(x^{-1}a^{-1}x^{2})^{2})x^{-2}ay^{-1}x^{-1}\\
			=& yay^{-1}a^{-1}x^{2}(y\circ (a^{-1}\circ x))x^{-2}ay^{-1}x^{-1}\\
			=& yay^{-1}a^{-1}x^{2}((y\circ a^{-1})\circ x)x^{-2}ay^{-1}x^{-1}\\
			=& yay^{-1}a^{-1}x^{2}(x^{-1}aya^{-2}x^{2})x^{-2}ay^{-1}x^{-1}\\
			=& yay^{-1}a^{-1}xaya^{-1}y^{-1}x^{-1}\\
			=& [y,a]x[a,y]x^{-1}\\
			=& [[y,a],x]\\
			=& [[a,y]^{-1},x]\\
			=& [[a,y],x]^{-1}, \text{as}\; [G,G,G]\subseteq Z(G)\\
			=& [x,[a,y]]\\
			=& 1, \text{as}\; a\in N_{\mu}({^\circ}G).
		\end{align*}
		\n Hence, $a\in N_{\lambda}({^\circ}G)$. Therefore, $N_{\mu}({^\circ}G)\subseteq N_{\lambda}({^\circ}G)$.
			
		\end{proof}

\begin{proposition}
	$N_{\lambda}({^\circ}G)$ is a normal subloop of the loop ${^\circ}G$. 
\end{proposition}
\begin{proof}
	
		\n $(i)$ Clearly, $(N_{\lambda}({^\circ}G)\circ x)\circ y = N_{\lambda}({^\circ}G)\circ (x\circ y)$ for all $x,y\in {^\circ}G$.
		
		\vspace{0.2 cm}
		
		\n ($ii$) Let $x\in {^\circ}G$ and $a\in N_{\lambda}({^\circ}G)$. Then 
		\begin{eqnarray*}
			x\circ a &=& a^{-1}xa^{2}\\
			&=& a^{-1}[x,a]axa\\
			&=& [x,a]xa,\; \text{as}\; [a^{-1},[x,a]] = 1\\
			&=& xax^{-1}a^{-1}xa\\
			&=&x^{-1}(x^{2}ax^{-1}a^{-1}xax^{-2})x^{2}\\
			&=& b\circ x, \text{where}\; 
		\end{eqnarray*}
	\n	where $b = x^{2}ax^{-1}a^{-1}xax^{-2}$. Since $N_{\lambda}({^\circ}G)$ is a normal subgroup of the group $G$, $b \in N_{\lambda}({^\circ}G)$. Thus, $x\circ N_{\lambda}({^\circ}G) = N_{\lambda}({^\circ}G)\circ x$, for all $x\in {^\circ}G$.
	
	\vspace{0.2 cm}
		
		\n ($iii$) Let $x,y\in {^\circ}G$ and $a\in N_{\lambda}({^\circ}G)$. Then
		\begin{eqnarray*}
		x\circ (y\circ a) &=& x\circ(a^{-1}ya^{2})\\
		&=& a^{-2}y^{-1}axa^{-1}yaya^{2}\\
		&=& a^{-1}y^{-1}(ya^{-1}y^{-1}axa^{-1}yay^{-1})y^{2}a^{2}\\
		&=& a^{-1}y^{-1}([y,a^{-1}]x[a^{-1},y]x^{-1})xy^{2}a^{2}\\
		&=& a^{-1}y^{-1}[[y,a^{-1}],x]xy^{2}a^{2}\\
		&=& b^{-1}y^{-1}xy^{2}b^{2}, \text{where}\; b = a[[y,a^{-1}],x]\\
		&=& (x\circ y)\circ b.
		\end{eqnarray*}
		Since $[y,a^{-1},x]\in Z(G) \subseteq N_{\lambda}({^\circ}G)$, $b\in N_{\lambda}({^\circ}G)$. Thus, $x\circ (y\circ N_{\lambda}({^\circ}G)) = (x\circ y)\circ N_{\lambda}({^\circ}G)$. Hence, $N_{\lambda}({^\circ}G)$ is a normal subloop.
\end{proof}

\begin{proposition}\label{munl}
	$N_{\mu}({^\circ}G)$ is a normal subloop of the loop ${^\circ}G$.
\end{proposition}
\begin{proof}
	Clearly, by the Proposition \ref{lmrn} $(i)$ and $(ii)$, $(N_{\mu}({^\circ}G)\circ x)\circ y = N_{\mu}({^\circ}G)\circ (x\circ y)$ and $x\circ (y\circ N_{\mu}({^\circ}G)) = (x\circ y)\circ N_{\mu}({^\circ}G)$ for all $x,y\in {^\circ}G$. Now, let $x\in {^\circ}G$ and $a\in N_{\mu}({^\circ}G)$. Then
	\begin{eqnarray*}
	x\circ a &=& a^{-1}xa^{2}\\
	&=& [a^{-1},x]xa\\
	&=& x[a^{-1},x]a, \text{because}\; [x,[a^{-1},x]] = 1\\
	&=& x^{-1}(x^{2}a^{-1}xax^{-1}ax^{-2})x^{2}\\
	&=& b\circ x,
	\end{eqnarray*}
\n	where $b= x^{2}a^{-1}xax^{-1}ax^{-2}$. Since $N_{\mu}({^\circ}G)$ is a normal subgroup of the group $G$, $b\in N_{\mu}({^\circ}G)$. Thus, $x\circ N_{\mu}({^\circ}G) = N_{\mu}({^\circ}G) \circ x$ for all $x\in {^\circ}G$. Hence, $N_{\mu}({^\circ}G)$ is a normal subloop of the loop ${^\circ}G$.
\end{proof}
\begin{corollary}\label{nunl}
	The nucleus $N({^\circ}G)$ is a normal subloop of the loop ${^\circ}G$. Moreover, $(N({^\circ}G), \circ)$ is group with the induced binary operation $\circ$.
\end{corollary}
\begin{proof}
	Since $N({^\circ}G) = N_{\lambda}({^\circ}G)\cap N_{\mu}({^\circ}G)\cap N_{\rho}({^\circ}G), N({^\circ}G) = N_{\mu}({^\circ}G) = N_{\rho}({^\circ}G)$ using the Proposition \ref{lmrn}. Hence, the corollary follows by the Proposition \ref{munl} and the Corollary \ref{nug}.
\end{proof}

\begin{lemma}\label{l1}
	View $C(^{\circ}G)$ as a subgroup of the group $G$. For $a\in C(^{\circ}G)$ and $x\in {^\circ}G$, we have 
	\begin{enumerate}
		\item[($i$)] $[a,x,x] = 1$ and $[a,x,a] = 1,$
		\item[($ii$)] $[x^{-1}, a^{-1}] = [a,x^{-1}]= [x,a],$
		\item[($iii$)] $[x^{2},a] = [x, a]^{2} = [x, a^{2}] = [x, a^{-1}] = [x^{-1},a],$
		\item[$(iv)$] $[a,x^{3}] = [a,x]^{3} = [a^{3},x] = 1$.
	\end{enumerate}
\end{lemma}
\begin{proof}
 Let $a\in C(^{\circ}G)$ and $x\in G$. Then, we have
\vspace{0.2 cm}

		\n ($i$) $[a,x,x] = [[a,x],x] = axa^{-1}xax^{-1}a^{-1}x^{-1}= ax(a^{-1}xa^{2})(a^{-1}x^{-1})^{2} = ax(x^{-1}ax^{2})(a^{-1}x^{-1})^{2} = a^{2}x^{2}x^{-2}a^{-2} = 1$. By similar argument, one can obtain $[a,x,a] = 1$.
		
		\vspace{0.2 cm}
		
		\n ($ii$) $[x^{-1}, a^{-1}] = x^{-1}a^{-1}xa = (x^{-1}a^{-1}x^{2})x^{-1}a = axa^{-2}x^{-1}a= ax(a^{-1}$ $xa^{2})^{-1} = axx^{-2}a^{-1}x = [a,x^{-1}]$. Similarly, $[a,x^{-1}] = [x,a]$.
		
	\vspace{0.2 cm}
		
		\n ($iii$) 	 
		\begin{eqnarray*}
		[x^{2}, a] &=& [x,[x,a]][x,a]^{2},\; \text{using the Lemma \ref{l3} $(i)$}\\
		&=& [x,[a,x]^{-1}][x,a]^{2}\\
		&=& [x,[a,x]]^{-1}[x,a]^{2}, \; \text{as}\; [G,G,G]\subseteq Z(G)\\
		&=& [[a,x],x][x,a]^{2}\\
		&=& [x,a]^{2}, \text{using part($i$)}
		\end{eqnarray*}
	 By the similar argument, $[x,a^{2}] = [x,a]^{2}$. Now, $[x^{2}, a] = x^{2}ax^{-2}$ $a^{-1} = x(a\circ x^{-1})a^{-1} = x(x^{-1}\circ a)a^{-1} = xa^{-1}x^{-1}a^{2}a^{-1} = xa^{-1}x^{-1}$ $a = [x,a^{-1}]$. By the similar argument, $[x,a^{2}] = [x^{-1},a]$. 
		\item[($iv$)] \begin{eqnarray*}
		[a,x^{3}] &=& [a,x^{2}][x^{2},[a,x]][a,x], \; \text{using the Lemma \ref{l3} $(ii)$}\\
		&=& [a,x^{2}][[a,x],x^{2}]^{-1}[a,x]\\
		&=& [a,x^{2}][a,x,x]^{-2}[a,x], \; \text{as}\; [G,G,G]\subseteq Z(G)\\
		&=& [a,x]^{3}, \; \text{using parts $(i)$ and $(iii)$}.
		\end{eqnarray*}
		By the similar argument, $[a^{3},x] = [a,x]^{3}$. Now, using parts $(ii)$ and $(iii)$, we have $[x,a]^{3} = [x,a]^{2}[x,a] = [x^{-1}, a][a,x^{-1}] = 1$. Thus, $[a,x^{3}] = 1$. Similarly, $[a^{3},x] = 1$.
	\end{proof}
\begin{corollary}
	$C(^{\circ}G)^{3}\subseteq Z(G)$.
\end{corollary}
\begin{proof}
	Follows directly from the Lemma \ref{l1} $(iv)$.
\end{proof}
\begin{proposition}\label{zg}
	$Z({^\circ}G) = C(^{\circ}G)\cap N_{i}({^\circ}G)$, where $i\in \{\lambda, \mu, \rho\}$.
\end{proposition}
\begin{proof}
	Let $x,y\in ^{\circ}G$. Then $a\in C({^\circ}G)\cap N_{\lambda}({^\circ}G)$
	
	$\iff(a\circ x)\circ y = a\circ (x\circ y)$
	
	$\iff (x\circ a)\circ y = (x\circ y)\circ a$
	
	$\iff y^{-1}a^{-1}xa^{2}y^{2} = a^{-1}y^{-1}xy^{2}a^{2}$
	
	$\iff (yay^{-1}a^{-1})x = x(y^{2}a^{2}y^{-2}a^{-2})$
	
	$\iff  [y,a]x = x[y^{2}, a^{2}]$
	
	$\iff [y,a]x = x[y,a]^{4}, \text{using the Lemma \ref{l1} $(iii)$},$
	
	$\iff [y,a]x = x[y,a], \text{using the Lemma \ref{l1} $(iv)$},$
	
	$\iff [a^{-1},y]x = x[a^{-1},y], \text{using the Lemma \ref{l1} $(ii)$},$
	
	$\iff a^{-1}yay^{-1}x = xa^{-1}yay^{-1}$
	
	$\iff y^{-1}xy = a^{-1}y^{-1}axa^{-1}ya$
	
	$\iff a^{-1}y^{-1}xy^{2}a^{2} = a^{-2}y^{-1}axa^{-1}yaya^{2}$
	
	$\iff (x\circ y) \circ a = (a^{-1}ya^{2})^{-1}x(a^{-1}ya^{2})^{2}$
	
	$\iff (x\circ y) \circ a = x\circ (y\circ a)$
	
	$\iff a\in C({^\circ}G)\cap N_{\rho}({^\circ}G)$.
	
	\n Thus, $C(^{\circ}G)\cap N_{\lambda}({^\circ}G)= C(^{\circ}G)\cap N_{\rho}({^\circ}G)$. Using the Proposition \ref{lmrn} $(i)$, $C({^\circ}G)\cap N_{\lambda}({^\circ}G) = C({^\circ}G)\cap N_{\mu}({^\circ}G)= C({^\circ}G)\cap N_{\rho}({^\circ}G)$. Hence, $Z({^\circ}G) = C({^\circ}G)\cap N({^\circ}G) = C(^{\circ}G)\cap N_{i}({^\circ}G)$, for any $i\in \{\lambda, \mu, \rho\}$.
\end{proof}

\n In \cite[Proposition 4.4, p. 1450058-10]{lal2014weak}, it is proved that if $3$ does not divide the order of $G$, then $C(^{\circ}G)^{2}\subseteq Z(G)$. Below, we prove that the commutant of $^{\circ}G$ and the center of the group $G$ are equal in this case.

\begin{theorem}\label{p6}
	Let $3$ does not divide the order of $G$. Then $C(^{\circ}G) = Z(G)$.
\end{theorem}
\begin{proof}
	Let $a\in C(^{\circ}G)$ and $x\in G$.  Then  $(xa)^{2}= a^{2}x^{2}$. Since $C(^{\circ}G)^{2}\subseteq Z(G)$, $(xa)^{2}= a^{2}x^{2}= x^{2}a^{2}$. This implies that $ax = xa$ for all $x\in G$. Hence $a\in Z(G)$. Therefore, $C(^{\circ}G)\subseteq Z(G)$. Since $Z(G) \subseteq C(^{\circ}G)$, $C(^{\circ}G) = Z(G)$.
\end{proof}

\begin{corollary}\label{p6c}
	Let $3$ does not divide the order of $G$. Then $C(^{\circ}G)=Z(G) = Z({^\circ}G)$.
\end{corollary}

\begin{proposition}\label{p4}
	 ${^\circ}G/ C(^{\circ}G)$ is a group.
\end{proposition}

\begin{proof}
	For all $\overline{x}, \overline{y}, \overline{z} \in {^\circ}G/C({^\circ}G)$, we have
	\begin{eqnarray*}
		\overline{x}\circ (\overline{y}\circ \overline{z}) &=& x\circ C({^\circ}G)\circ ((y\circ C({^\circ}G))\circ (z\circ C({^\circ}G)))\\
		&=& (x\circ (y\circ z))\circ C({^\circ}G)\\
		&=& (x\circ z^{-1}yz^{2})\circ C({^\circ}G)\\
		&=& z^{-2}y^{-1}zxz^{-1}yz^{2}z^{-1}yz^{2}\circ C({^\circ}G)\\
		&=& z^{-1}(y^{-1}y)z^{-1}y^{-1}zxz^{-1}yz(y^{-1}y)yz^{2}\circ C({^\circ}G)\\
		&=& z^{-1}y^{-1}(yz^{-1}y^{-1}zxz^{-1}yzy^{-1})y^{2}z^{2}\circ C({^\circ}G)\\
		&=& z^{-1}y^{-1}[y,z^{-1}]x[z^{-1},y]y^{2}z^{2}\circ C({^\circ}G)\\
		&=& z^{-1}y^{-1}[[y,z^{-1}],x]xy^{2}z^{2}\circ C({^\circ}G)\\
		&=& z^{-1}(y^{-1}xy^{2})z^{2}\circ C({^\circ}G), \text{as} \;[y,z^{-1},x]\in Z(G)\subseteq C({^\circ}G)\\
		&=& ((x\circ y)\circ z)\circ C({^\circ}G)\\
		&=& (\overline{x}\circ \overline{y})\circ \overline{z}.
	\end{eqnarray*}
	Hence, ${^\circ}G/ C({^\circ}G)$ is a group.
\end{proof}
	\begin{theorem}\label{p3}
		$({^\circ}G/C(^{\circ}G), \circ)=({^\circ}(G/C(^{\circ}G)), \circ)$.
	\end{theorem}
	\begin{proof}
		Since $G$ is a nilpotent group of class 3, $G/Z(G)$ is a nilpotent group of class 2. Now, define a group homomorphism $\phi: G/Z(G)\longrightarrow G/C(^{\circ}G)$ by $Z(G) x \mapsto C(^{\circ}G) x$. Then, one can easily observe that $\phi$ is a surjective homomorphism. Hence, $G/C(^{\circ}G)$ being the homomorphic image of the group $G/Z(G)$ is a nilpotent group of class at most 2. Therefore, by \cite[Theorem 3.6]{foguel2001gyrogroups}, the associated right gyrogroup ${^\circ}(G/C(^{\circ}G))$ is a group.
		
\vspace{0.2 cm}
		
\n Let $x, y\in {^\circ}G$. Since $C(^{\circ}G)$ is a normal subgroup of $G$, $C(^{\circ}G)\circ x = x^{-1} C(^{\circ}G) x^{2} = (x^{-1} C(^{\circ}G) x)x=C(^{\circ}G) x$. Therefore, $(C(^{\circ}G)\circ x)\circ (C(^{\circ}G)\circ y) = C(^{\circ}G)\circ (x\circ y) = C(^{\circ}G)(x\circ y)=(C(^{\circ}G)x)\circ (C(^{\circ}G)y)$. Thus, $({^\circ}G/C(^{\circ}G), \circ)=({^\circ}(G/C(^{\circ}G), \circ)$.
	\end{proof}
	
	\n Let $3$ does not divide the order of $G$. Then, by Corollary \ref{p6c} and Theorem \ref{p3}, the exact sequence 
\begin{eqnarray}\label{s5e1}
	\begin{tikzcd}
		0 \arrow[r] & Z(G) \arrow[r, "i"] & G \arrow[r, "\pi"] & G/Z(G)\arrow[r] & 1 
	\end{tikzcd}
\end{eqnarray}

\n of the groups induces the exact sequence

\begin{eqnarray}\label{s5e2}
	\begin{tikzcd}
		0 \arrow[r] & Z(^{\circ}G) \arrow[r, "i"] & ^{\circ}G \arrow[r, "\pi"] & {^\circ}(G/Z(^{\circ}G))\arrow[r] & 1 
	\end{tikzcd}
\end{eqnarray}

\n of the loops. Note that, $Z(^{\circ}G)$ and $G/Z(^{\circ}G)$ are groups. Now, $G$ can be identified with
\begin{equation}\label{s5e3}
 \{(a, \overline{x}) \mid a \in Z(G), \overline{x}\in G/Z(G) \},
\end{equation} 
\n with the binary operation defined for all $(a, \overline{x}), (b, \overline{y})\in G$ as,
\begin{equation}\label{s5e4}
(a, \overline{x})\cdot(b, \overline{y}) = (abf(\overline{x},\overline{y}), \overline{x}\overline{y}),
\end{equation}

\n where $f: G/Z(G)\times G/Z(G)\longrightarrow Z(G)$ is a normalized function given by the equation 
\begin{equation}\label{s2e5}
(\overline{x}\cdot \overline{y})\cdot \overline{z} = \overline{x}f(\overline{y},\overline{z})\cdot (\overline{y}\cdot \overline{z}),\;\; \forall\; \overline{x}, \overline{y}, \overline{z} \in G/Z(G).
\end{equation}

\n Let $\mathbf{Z}^{2}(G/Z(G), G)$ and $\mathbf{B}^{2}(G/Z(G), Z(G))$ denotes the group of all the $2$-cocycles and $2$-coboundaries associated to the group extension $G$ of the group $Z(G)$ by the group $G/Z(G)$. Then, for all $(a, \overline{x}), (b, \overline{y}) \in {^{\circ}G}$,
\begin{eqnarray*}
	(a, \overline{x})\circ (b, \overline{y}) &=& (b, \overline{y})^{-1}\cdot(a, \overline{x})\cdot (b, \overline{y})^{2}\\  &=& (b^{-1}f(\overline{y}, \overline{y}^{-1})^{-1}, \overline{y}^{-1})\cdot(a, \overline{x})\cdot(b^{2}f(\overline{y}, \overline{y}), \overline{y}^{2})\\ &=& (b^{-1}af(\overline{y}, \overline{y}^{-1})^{-1}
	f(\overline{y}^{-1}, \overline{x}), \overline{y}^{-1}\overline{x})\cdot(b^{2}f(\overline{y}, \overline{y}), \overline{y}^{2})\\ &=& (b^{-1}ab^{2}f(\overline{y}, \overline{y}^{-1})^{-1}f(\overline{y}^{-1}, \overline{x})f(\overline{y}, \overline{y})f(\overline{y}^{1}\overline{x}, \overline{y}^{2}), \overline{y}^{-1}\overline{x}\overline{y}^{2}).
\end{eqnarray*}

\n By \cite[Section 2]{ddv}, if
\begin{eqnarray}\label{s5e6}
	\begin{tikzcd}
		0 \arrow[r] & Z(L) \arrow[r, "i"] & L \arrow[r, "\pi"] & Q \arrow[r] & 1 
	\end{tikzcd}
\end{eqnarray}

\n is the central extension of the loop $L$ by the loop $Q$, then $L$ is identified with $Z(L)\times Q$ with the binary operation
\begin{equation}\label{s5e7}
(a, \overline{x})\cdot(b, \overline{y}) = (ab\phi(\overline{x},\overline{y}), \overline{x}\overline{y}),
\end{equation}

\n where $\phi:Q\times Q \rightarrow Z(L)$ is a 2-cocycle.

Thus, we define
\begin{center}
	$(a, \overline{x})\circ (b, \overline{y}) = (a\circ b\circ {^{\circ}f(\overline{x}, \overline{y})}, \overline{x}\circ \overline{y}),$
\end{center} 
where $a\circ b = b^{-1}ab^{2}$, $\overline{x}\circ \overline{y} = \overline{y}^{-1}\overline{x}\overline{y}^{2}$ and ${^{\circ}f(\overline{x}, \overline{y})} = f(\overline{y}, \overline{y}^{-1})^{-1}f(\overline{y}^{-1}, \overline{x})$ $f(\overline{y}, \overline{y})f(\overline{y}^{-1}\overline{x}, \overline{y}^{2})$, for all $a,b\in Z({^\circ}G)$ and $\overline{x}, \overline{y} \in {^\circ}(G/Z({^\circ}G))$. Since, $G$ is a nilpotent group of class 3, the associated right gyrogroup $({^\circ}G, \circ)$ is actually a gyrogroup.
\n Let $f,g \in \mathbf{Z}^{2}(G/Z(G), Z(G))$ be two 2-cocycles associated to the group extension $G$. Then there exists a map $\tau : G/Z(G) \longrightarrow Z(G)$ such that  
\begin{equation}\label{e7}
g(\overline{x}, \overline{y}) = \tau(\overline{x})\cdot \tau(\overline{y})\cdot f(\overline{x}, \overline{y})\cdot \tau(\overline{x} \cdot \overline{x})^{-1}
\end{equation}
for all $\overline{x},\overline{y} \in G/Z(G)$. Now, let ${^\circ}f, {^\circ}g$ be $2$ - cocycles associated to the loop extension ${^\circ}G$ of the group $Z(G)$ by the group ${^\circ}(G/Z(G))$. Then,
\begin{eqnarray*}
	{^{\circ}f(\overline{x}, \overline{y})} = f(\overline{y}, \overline{y}^{-1})^{-1}f(\overline{y}^{-1}, \overline{x})f(\overline{y}, \overline{y})f(\overline{y}^{-1}\overline{x}, \overline{y}^{2})\\
	\text{and}\;
	{^{\circ}g(\overline{x}, \overline{y})} = g(\overline{y}, \overline{y}^{-1})^{-1}g(\overline{y}^{-1}, \overline{x})g(\overline{y}, \overline{y})g(\overline{y}^{-1}\overline{x}, \overline{y}^{2}).
\end{eqnarray*}
Since, $Z(G)$ is the center of the group $G$, using the Equation \ref{e7}, we get
\begin{equation}\label{e5}
{^\circ}g(x, y) = \tau(\overline{x})\circ \tau(\overline{y})\circ {{^\circ}f}(\overline{x}, \overline{y})\circ (\tau(\overline{x} \cdot \overline{y}))^{-1}.
\end{equation}
Thus, we define a map ${^\circ}\tau : {^\circ}(G/Z(G))\longrightarrow Z(G)$ by ${^\circ}\tau(\overline{x}) = \tau(\overline{x})$ for all $\overline{x}\in {^\circ}(G/Z(G))$. We will denote the map ${^\circ}\tau$ by the map $\tau$.

\vspace{.2cm}

\n Let ${^\circ}\mathbf{Z}^{2}({^{\circ}(G/Z(G))}, Z(G))$(we will write ${^\circ}\mathbf{Z}^{2}$ in short) be the collection of all associated 2-cocycles ${^\circ}f$ associated to the loop extensions ${^{\circ}G}$ of the normal subloop $Z(G)$ by the group ${^{\circ}(G/Z(G))}$. Then we define a relation $\sim$ on the set ${^\circ}\mathbf{Z}^{2}$ as, for any two 2-cocycles ${^{\circ}f}, {^{\circ}g}\in {^\circ}\mathbf{Z}^{2}$, we say that ${^{\circ}f}\sim {^{\circ}g}$ if there exists a normalized map $\tau: {^{\circ}(G/Z(G))}\longrightarrow Z(G)$ satisfying the Equation \ref{e5}. One can easily observe that $\sim$ is an equivalence relation on the set ${^\circ}\mathbf{Z}^{2}$. Let ${^\circ \mathbf{H}^{2}}({^{\circ}(G/Z(G))}, Z(G)) =$ \{$[{^{\circ}f}]\mid {^{\circ}f}\in {^\circ}\mathbf{Z}^{2}$\} be the set of all the equivalence classes $[{^{\circ}f}]$ of the elements in ${^\circ}\mathbf{Z}^{2}$. One can observe that, ${^\circ}\mathbf{H}^{2}\subseteq \mathbf{H}^{2}$.

	
\section{Nilpotency Class of the Associated Gyrogroup ${^\circ}G$}
\n In this section, we will find the niplotency class of the associated gyrogroup ${^\circ}G$. It was proved in \cite[Proposition 4.5, 1450058-11]{lal2014weak} that a group $G$ is a nilpotent group of class $2$ if and only if $^{\circ}G$ is a nilpotent group of class atmost $2$. It should be noted that the word ``atmost" is missed in that proposition. 

\begin{theorem}\label{p5}
	Let $G$ be a group such that $3$ does not divide the order of $G$. Then $G$ is a nilpotent group of class exactly $2$ if and only if ${^{\circ}G}$ is a nilpotent group of class exactly $2$.
\end{theorem}
\begin{proof}
	Let $G$ be a nilpotent group of class exactly $2$ such that $3$ does not divide the order of $G$. Then the associated right gyrogroup ${^{\circ}G}$ is a group of class atmost $2$. If class of ${^{\circ}G}$ is $1$, then $Z({^{\circ}G})={^{\circ}G}$. By Corollary \ref{p6c}, $Z(G)=Z({^{\circ}G})$. This means that $G=Z(G)$. Therefore, $G$ is of class $1$. Conversely, let ${^\circ}G$ be a nilpotent group of class $2$. Then, by \cite[Theorem 3.6]{foguel2001gyrogroups}, $G$ is a nilpotent group of class $2$.
\end{proof}

\begin{theorem}\label{4.2}
Let $G$ be a group such that $3$ does not divide the order of $G$. Then $G$ is a nilpotent group of class $3$ if and only if ${^{\circ}G}$ is a nilpotent loop of class $3$.
\end{theorem}
\begin{proof}
Let $G$ be a nilpotent group of class $3$ such that $3$ does not divide the order of $G$. Then $G/Z(G)$ is a  nilpotent group of class $2$.
By Theorem \ref{p5}, the associated right gyrogroup ${^{\circ}(G/Z(G))}$ is a nilpotent group of class $2$. Also, by the Theorem \ref{p3} and Corollary \ref{p6c}, ${^{\circ}G}/Z({^{\circ}G})$ is a group of class $2$. Hence, ${^{\circ}G}$ is a nilpotent loop of class $3$. Conversely, let ${^{\circ}G}$ be a nilpotent loop of class $3$. Then, ${^{\circ}G}/Z({^{\circ}G}) = ^{\circ}(G/Z(G))$ is a nilpotent group of class $2$. Hence, by \cite[Theorem 3.6]{foguel2001gyrogroups}, $G/Z(G)$ is a nilpotent group of class $2$. Therefore, $G$ is a nilpotent group of class $3$.
\end{proof}

	\begin{theorem}\label{cl}
	If $G$ is a nilpotent group of class $3$, then ${^\circ}G$ is nilpotent loop of class 2 if and only if $[x,y]^{3}\in C(^{\circ}G)$ for all $x,y \in G$.
\end{theorem}
\begin{proof}
	Let $x, y\in {^\circ}G$. Since $x\circ y = {^\circ}[x, y]\circ (y\circ x)$, 
	\begin{eqnarray*}
		{^\circ}[x,y] &=& (x\circ y)\circ (y\circ x)^{-1}\\
		&=& y^{-1}xy^{2}\circ (x^{-1}yx^{2})^{-1}\\
		&=& x^{-1}yx^{2}y^{-1}xy^{2}(x^{-2}y^{-1}x)^{2}\\
		&=& x^{-1}yx^{2}y^{-1}xy^{2}x^{-2}y^{-1}x^{-1}y^{-1}x\\
		&=& (x^{-1}yx)(xy^{-1})(xy^{2}x^{-2}y^{-1})(x^{-1}y^{-1}x)\\
		&=& (x^{-1}yxy^{-1})(yxy^{-1}x^{-1})(x^{2}y^{2}x^{-2}y^{-2})(yx^{-1}y^{-1}x)\\
		&=& [x^{-1}, y][y,x][x^{2}, y^{2}][y,x^{-1}]\\
		&=& [y,x][x^{2},y^{2}], \text{as $[G,G]$ is abelian}\\
		&=& [y,x][x^{2},y]^{2}[y,[x^{2},y]], \text{by Lemma \ref{l3} ($ii$) and}\; [G,G,G]\subseteq Z(G)
		\\
		&=& [y,x]([x,[x,y]][x,y]^{2})^{2}[y, [x,[x,y]][x,y]^{2}], \text{using Lemma \ref{l3} ($i$)},\\
		&=& [y,x][x,[x,y]]^{2}[x,y]^{4}[y,[x,[x,y]]][[x,[x,y]], [y, [x,y]^{2}]][y, [x,y]^{2}]\\
		&=& [y,x][x,[x,y]]^{2}[x,y]^{4}[y, [x,y]^{2}].
	\end{eqnarray*}
Thus,
\begin{equation}\label{ec}
{^\circ}[x,y] = [x,y]^{3}[x,[x,y]]^{2}[y, [x,y]]^{2}.
\end{equation}
Now, ${^\circ}G$ is nilpotent loop of class 2 $\iff {^\circ}G/Z({^\circ}G)$ is abelian $\iff {^\circ}[x,y]\in Z({^\circ}G)$ for all $x,y\in {^\circ}G$. Since $G$ is a nilpotent group of class 3, $[G,G,G]\subseteq Z(G)$ and $Z(G)\subseteq Z({^\circ}G)\subseteq C(^{\circ}G)$, ${^\circ}G$ is nilpotent loop of class 2 $\iff [x,y]^{3}\in C(^{\circ}G)$.
\end{proof}
\begin{corollary}\label{2e}
	Let $G$ be a 2- Engel group. Then, ${^\circ}G$ is a nilpotent loop of class $2$.
\end{corollary}
\begin{proof}
		Let $G$ be a 2 - Engel group. Then, $G$ is of class atmost $3$ and $[x,y,z]^{3} = 1$ for all $x,y,z \in G$. Therefore, $[[x,y]^{3},z] = 1$ which implies that $[x,y]^{3}\in Z(G)$. Thus, $[x,y]^{3}\in C(^{\circ}G)$ and the corollary follows from the Theorem \ref{cl}. 
\end{proof}
\begin{corollary}
	Let $G$ be a group of exponent $3$, then ${^\circ}G$ is a nilpotent loop of class $2$.
\end{corollary}
\begin{proof}
	Follows immediately from the Corollary \ref{2e} and the fact that every group of exponent 3 is a 2 - Engel group.
\end{proof}


\section{Problem of Abelian Inner Mapping Groups}

\n It was an open problem whether there exists a loop of nilpotency class $3$ with abelian inner mapping group. Csorgo in \cite{pcos} gave its answer in affirmative by giving a loop of order $2^7$. This problem is still open for odd case, that is whether there exists an odd order loop of nilpotency class $3$ whose inner mapping group is abelian. In this section, we investigate when one hopes to find its answer of this problem for the loop of order $3^n$ for some positive integer $n$.

\n In a talk of the first conference on Artificial Intelligence and Theorem Proving in the year 2016, M. Kinyon along with B. Veroff gives the following theorem (see http://aitp-conference.org/2016/slides/Kinyon \textunderscore Obergurgl.pdf)

\begin{theorem} (Unpublished)\label{knth}
	Let $Q$ be a loop. Then
	\begin{itemize}
		\item[$(i)$] If $Q/N(Q)$ is an abelian group, $Q/Z(Q)$ is a group and $K(\cdot, \cdot)$ is associative, then $Inn(Q)$ is abelian. 
		\item[($ii$)] If $Inn(Q)$ is abelian, then the loop commutator is associative. 
	\end{itemize}
	\end{theorem} 
	
\n  Therefore, we would like to get the conditions of Theorem  \ref{knth} satisfied to find the answer in affirmative. 

\begin{proposition}\label{asct}
$\mathcal{A}({^\circ}G)\subseteq Z(G)$.
\end{proposition}
\begin{proof}
	Let $x,y,z\in {^\circ}G$. Then, for $A(x,y,z)\in \mathcal{A}({^\circ}G)$ we have,
	\begin{eqnarray*}	
	A(x,y,z) &=& ((x\circ y)\circ z)\circ (x\circ(y\circ z))^{-1}\\
	&=& (z^{-1}y^{-1}xy^{2}z^{2})\circ (z^{-2}y^{-1}zxz^{-1}yzyz^{2})^{-1}\\
	&=& (z^{-2}y^{-1}zxz^{-1}yzyz^{2})(z^{-1}y^{-1}xy^{2}z^{2})(z^{-2}y^{-1}zxz^{-1}yzyz^{2})^{-2}\\
	&=& z^{-2}y^{-1}zxz^{-1}yzyzy^{-1}xyz^{-1}y^{-1}zx^{-1}z^{-2}y^{-1}zx^{-1}z^{-1}yz^{2}\\
	&=& z^{-2}y^{-1}zxz^{-1}yzyzy^{-1}x[y,z^{-1}]x^{-1}z^{-2}y^{-1}zx^{-1}z^{-1}yz^{2}\\
	&=& z^{-2}y^{-1}zxz^{-1}yz^{2}[[z^{-1}, y], x]z^{-2}y^{-1}zx^{-1}z^{-1}yz^{2}\\
	&=& [[z^{-1}, y], x], \text{because}\; [[z^{-1}, y], x]\in Z(G).
	\end{eqnarray*}
	Thus, $A(x,y,z) = [[z^{-1}, y], x]\in Z(G),$ for all $x,y,z\in {^\circ}G$. Hence,  $\mathcal{A}({^\circ}G)\subseteq Z(G)$.
\end{proof}

\begin{proposition}\label{GNab}
	$({^\circ}G/N({^\circ}G), \circ)$ is an abelian group.
\end{proposition}
\begin{proof}
	By the Proposition \ref{asct}, $\mathcal{A}({^\circ}G)\subseteq Z(G)$ and the fact that $Z(G)\subseteq N({^\circ}G)$, $({^\circ}G/N({^\circ}G), \circ)$ is a group. Since $G$ is of class $3$, for all $u,v,x,y\in {^\circ}G$, 
	
		\[[x,[[u,v],y]] = 1.\]
		
		\n Therefore, $[u,v]\in N_{\mu}({^\circ}G) = N({^\circ}G)$ for all $u,v \in {^\circ}G$. Hence, by the Equation \ref{ec}, ${^\circ}[u,v]\in N({^\circ}G)$. Thus, $({^\circ}G/N({^\circ}G), \circ)$ is an abelian group.
\end{proof}

\begin{proposition}\label{GZgp}
	$({^\circ}G/Z({^\circ}G), \circ)$ is a group.
\end{proposition}
\begin{proof}
	Follows directly form the fact that $Z(G)\subseteq Z({^\circ}G)$ and the Proposition \ref{asct}.
\end{proof}

\begin{proposition}\label{comm}
	The commutator operation $^{\circ}[\cdot, \cdot]: {^\circ}G \times {^\circ}G \longrightarrow {^\circ}G$ is associative if and only if $[[x,y],z]^{9} = [x,[y,z]]^{9}$ for all $x,y,z\in {^\circ}G$.
\end{proposition}
\begin{proof}
	Let $x,y,z\in {^\circ}G$. Then by the Equation \ref{ec}, we have
	\begin{equation*}
	{^\circ}[{^\circ}[x,y],z] = [{^\circ}[x,y],z]^{3}[{^\circ}[x,y], [{^\circ}[x,y],z]]^{2}[z, [{^\circ}[x,y],z]]^{2}
	\end{equation*}
\n	Now, using the fact that $G$ is nilpotent group of class 3 and the Lemma \ref{l3} $(i)$, $[{^\circ}[x,y],z] = [[x,y]^{3},z]$. Since, $[x,y,z]\in Z(G), [[x,y]^{3},z] = [[x,y],z]^{3}$. Therefore, we get
	\begin{equation*}
	{^\circ}[{^\circ}[x,y],z] = [[x,y],z]^{9}.
	\end{equation*} 
	
	\n By the similar argument, we have	
	\begin{equation*}
	{^\circ}[x,{^\circ}[y,z]] = [x,[y,z]]^{9}.
	\end{equation*}
	
	Thus, the commutator operation $[\cdot, \cdot]$ is associative
	\begin{eqnarray*}
	\iff {^\circ}[{^\circ}[x,y],z] =  {^\circ}[x,{^\circ}[y,z]] \iff [[x,y],z]^{9} = [x,[y,z]]^{9},
	\end{eqnarray*}
	
		\n for all $x,y,z\in {^\circ}G$
\end{proof}

\begin{proposition}\label{innnab}
	Let $G$ be a group such that $3$ does not divide the order of the group $G$. Then the commutator operation $^{\circ}[\cdot, \cdot]: {^\circ}G \times {^\circ}G \longrightarrow {^\circ}G$ is not associative.
\end{proposition}
\begin{proof}
	On contrary, suppose that the commutator operation $^{\circ}[\cdot, \cdot]$ is associative. Therefore, by the Proposition \ref{comm}, $[[x,y],z]^{9} = [x,[y,z]]^{9}$ for all $x,y,z\in G$. Since, 3 does not divide the order of the group $G$, 9 does not divide the order of the group $G$. Therefore, there is an isomorphism from $G$ to $G$ given by $x\hookrightarrow x^{9}$ which gives $[[x,y],z] = [x,[y,z]]$. Thus, the commutator operation $[\cdot, \cdot]: G\times G \longrightarrow G$ is associative. This is a contradiction, by the Levi's Theorem  \cite{levi1942groups}. 
\end{proof}
\begin{theorem}
	Let $G$ be a group such that $3$ does not divide the order of the group $G$. Then $Inn({^\circ}G)$ is not abelian.
\end{theorem}
\begin{proof}
	On contrary, suppose that $Inn({^\circ}G)$ is abelian. Then by the Theorem \ref{knth} $(ii)$, the commutator operation ${^\circ}[\cdot, \cdot]$ is associative. This is a contradiction, by the Proposition \ref{innnab}.
\end{proof}

\n Note that, if $G_1$ and $G_2$ are two groups, then $^{\circ}(G_1\times G_2)=^{\circ}G_1 \times ^{\circ}G_2$. Since $G$ is nilpotent, it is sufficient to discuss about the associated gyrogroup $^{\circ}G$ for $3$-groups $G$ of nilpotency class $3$. According to Theorem \ref{cl},  ${^\circ}G$ is nilpotent loop of class 2 if and only if $[x,y]^{3}\in C(^{\circ}G)$ for all $x,y \in G$. If $[G,G]$ is of exponent $3$, then by the Theorem \ref{cl}, ${^\circ}G$ is nilpotent loop of class 2. Therefore, if there is a $3$-group for which $[x,y]^{3}\notin C(^{\circ}G)$, exponent of $[G,G]$ is not $3$ and $[[x,y],z]^{9} = [x,[y,z]]^{9}$ for all $x,y,z\in G$, then one can hope to get a loop of class $3$ with abelian inner mapping group.

\vspace{0.2 cm}

\n \textbf{Acknowledgment:} The first author is supported by the Junior Research Fellowship of UGC, India.


\end{document}